\documentclass[11pt]{amsart}

\usepackage{a4wide}
\usepackage{amssymb}
\usepackage{mathrsfs}
\usepackage{enumerate}
\usepackage{esint}
\usepackage{titletoc}
\usepackage[colorlinks=true, urlcolor=blue, linkcolor=blue, citecolor=black]{hyperref}
\usepackage[nameinlink]{cleveref}
\usepackage{graphicx}
\usepackage{todonotes}

\theoremstyle{plain}
\newtheorem{theorem}{Theorem}[section]
\newtheorem{lemma}[theorem]{Lemma}
\newtheorem{corollary}[theorem]{Corollary}

\newtheorem{assumption}{Assumption}

\theoremstyle{definition}

\newtheorem{remark}[theorem]{Remark}

\numberwithin{equation}{section}

\renewcommand{\d}{\textnormal{d}}
\newcommand{\N}{\mathbb{N}}
\newcommand{\R}{\mathbb{R}}

\DeclareMathOperator{\supp}{supp}

\newcommand{\Dsp}{\mathcal{D}^{\vec{s},\vec{p}}(\mathbb{R}^n)}
\DeclareMathOperator{\dist}{dist}

\makeatletter
\@namedef{subjclassname@2020}{%
  \textup{2020} Mathematics Subject Classification}
\makeatother

\newcommand{\BIGOP}[1]
{
\mathop{\mathchoice%
{\raise-0.22em\hbox{\huge $#1$}}%
{\raise-0.05em\hbox{\Large $#1$}}{\hbox{\large $#1$}}{#1}}}
\newcommand{\bigtimes}{\BIGOP{\times}}
\newcommand{\eps}{\varepsilon}

\title[]{The concentration-compactness principle for the nonlocal anisotropic $p$-Laplacian of mixed order}

\author{Jamil Chaker}
\address{Fakult\"at f\"ur Mathematik, Universit\"at Bielefeld, 33615 Bielefeld, Germany}
\email{jchaker@math.uni-bielefeld.de}

\author{Minhyun Kim}
\address{Fakult\"at f\"ur Mathematik, Universit\"at Bielefeld, 33615 Bielefeld, Germany}
\email{minhyun.kim@uni-bielefeld.de}

\author{Marvin Weidner}
\address{Fakult\"at f\"ur Mathematik, Universit\"at Bielefeld, 33615 Bielefeld, Germany}
\email{mweidner@math.uni-bielefeld.de}

\thanks{Jamil Chaker gratefully acknowledges support by the DFG through CRC 1283. Minhyun Kim and Marvin Weidner gratefully acknowledge financial support by the DFG through IRTG 2235.}

\begin{document}

\begin{abstract}
In this paper, we study the existence of minimizers of the Sobolev quotient for a class of nonlocal operators with an orthotropic structure having different exponents of integrability and different orders of differentiability. Our method is based on the concentration-compactness principle which we extend to this class of operators. One consequence of our main result is the existence of a nontrivial nonnegative solution to the corresponding critical problem.
\end{abstract}

\maketitle


\section{Introduction} \label{sec:introduction}
In this paper, we study variational problems for the nonlocal anisotropic $p$-Laplacian of mixed order. For $x\in\R^n$, let
\begin{equation}\label{eq:critintro}
\sum_{i=1}^n (-\partial_{ii})^{s_i}_{p_i} u(x) := \sum_{i=1}^n s_i(1-s_i) \int_{\mathbb{R}} \frac{|u(x)-u(x+he_i)|^{p_i-2}(u(x)-u(x+he_i))}{|h|^{1+s_ip_i}} \,\d h,
\end{equation}
where $s_1,\dots,s_n\in(0,1)$ and $p_1,\dots,p_n>1$.  
This operator can be seen as a nonlocal analog of the anisotropic $p$-Laplacian
\begin{equation}\label{eq:locaniso}
\sum_{i=1}^n (-\partial_{ii})_{p_i} u(x) = -\sum_{i=1}^n \frac{\partial}{\partial x_i} \left( \left|\frac{\partial u(x)}{\partial x_i}\right|^{p_i-2}\frac{\partial u(x)}{\partial x_i} \right).
\end{equation}
Local operators with such an orthotropic structure are well known in the literature and there are several results related to this type of operators, see \cite{Lions69,Mar89,BMS90,DC09,DCM09,FS19,BB20,CFS21,dST21,FVV21, BBLV21} and the references therein. We can read off \eqref{eq:critintro}, that the operator under consideration has on the one hand different exponents of integrability and on the other hand different orders of differentiability. Considering such operators, it is natural to work on anisotropic Sobolev spaces.
The aim of this paper is to study the existence of minimizers of the Sobolev quotient for the operator in \eqref{eq:critintro}. As a consequence of the existence of a minimizer, we get the existence of a solution to the corresponding critical problem. One of the main auxiliary results in this paper is a robust Sobolev-type inequality for the operator $\sum_{i=1}^n (-\partial_{ii})^{s_i}_{p_i}$.

The approach we use in this paper is based on the concentration-compactness principle (CCP) that provides an important tool to prove relative compactness of minimizing sequences. It has been introduced by P.-L. Lions in a series of papers, see \cite{Lions1,Lions2,Lions3,Lions4}.  
Before we address known results from the literature and describe the strategy of our paper, we formulate the main results and assumptions of the present work.

Given $s_1,\dots,s_n\in(0,1)$ and $p_1,\dots,p_n>1$, we define
\begin{equation*}
\bar{s} = \left( \frac{1}{n} \sum_{i=1}^n \frac{1}{s_i} \right)^{-1}, \quad \overline{sp}=\left( \frac{1}{n} \sum_{i=1}^n \frac{1}{s_i p_i} \right)^{-1}, \quad\text{and}\quad p^{\ast} = \frac{n\overline{sp}/\bar{s}}{n-\overline{sp}}.
\end{equation*}
\begin{assumption}\label{assumption:sp}
Given $s_1,\dots,s_n\in(0,1)$ and $p_1,\dots,p_n>1$, we assume
\[ \overline{sp} < n\quad \text{ and } \quad  p_{\max} := \max\lbrace p_1, \dots, p_n \rbrace < p^{\ast}. \]
\end{assumption}
Let $\vec{s}=(s_1,\dots,s_n)$ and $\vec{p}=(p_1,\dots,p_n)$. 
We introduce the homogeneous anisotropic Sobolev space $\mathcal{D}^{\vec{s}, \vec{p}}(\mathbb{R}^n)$ as
\begin{equation*}
\mathcal{D}^{\vec{s}, \vec{p}}(\mathbb{R}^n) = \left\lbrace u \in L^{p^\ast}(\mathbb{R}^n): \sum_{i=1}^n \int_{\mathbb{R}^n} \int_{\mathbb{R}} s_i(1-s_i) \frac{|u(x)-u(x+he_i)|^{p_i}}{|h|^{1+s_i p_i}} \,\d h \,\d x < \infty \right\rbrace
\end{equation*}
equipped with the norm $\Vert u \Vert_{\Dsp}$, where
\begin{equation*}
\Vert u \Vert_{\Dsp} := \sum_{i=1}^n \Vert D^{s_i}_{p_i} u \Vert_{L^{p_i}(\R^n)} := \sum_{i=1}^n \left( \int_{\mathbb{R}^n} \int_{\mathbb{R}} s_i(1-s_i) \frac{|u(x)-u(x+he_i)|^{p_i}}{|h|^{1+s_i p_i}} \,\d h\, \d x \right)^{1/p_i}
\end{equation*}
with
\begin{equation*}
|D^{s_i}_{p_i} u(x)|^{p_i} := s_i(1-s_i) \int_{\mathbb{R}} \frac{|u(x)-u(x+he_i)|^{p_i}}{|h|^{1+s_i p_i}} \,\d h
\end{equation*}
for $i=1,\dots,n$. Note that $\Dsp$ becomes a normed space in the light of \Cref{thm:sobolev}. It is the natural space for minimizing the functional connected to the Sobolev quotient. The first main result of this paper is a robust Sobolev-type inequality.

\begin{theorem} \label{thm:sobolev}
Let $s_1,\dots,s_n\in[s_0,1)$ for some $s_0\in(0,1)$ and $p_1,\dots,p_n>1$ be such that \Cref{assumption:sp} holds.
Then there is a constant $C = C(n, p^{\ast}, p^{\ast}-p_{\max}, s_0) > 0$ such that
for every $u \in \Dsp$
\begin{equation*}
\|u\|_{L^{p^{\ast}}(\mathbb{R}^n)} \leq C\Vert u \Vert_{\Dsp}.
\end{equation*}
\end{theorem}

\Cref{thm:sobolev} is robust in the sense that the appearing constant depends only on a lower bound $s_0 \le s_1,\dots,s_n$. This allows us to recover the local anisotropic Sobolev inequality (see \cite{Troi69,Rak79,Rak80,KK83})
\begin{equation} \label{eq:Troisi}
\|u\|_{L^q(\mathbb{R}^n)} \leq C \sum_{i=1}^n \|D_i u\|_{L^{p_i}(\mathbb{R}^n)},
\end{equation}
where
\begin{equation} \label{eq:p}
p_i > 1, \quad \bar{p}=\left( \frac{1}{n} \sum_{i=1}^n \frac{1}{p_i} \right)^{-1}, \quad \bar{p}<n, \quad q = \frac{n\bar{p}}{n-\bar{p}}, \quad \text{and}\quad p_{\max} < q.
\end{equation}
Moreover, our result provides an alternative proof of \eqref{eq:Troisi}.

Our second main result establishes the existence of nonnegative minimizers of the Sobolev quotient.

\begin{theorem} \label{thm:existence}
Assume $\vec{s}=(s_1,\dots,s_n)$ and $\vec{p}=(p_1,\dots,p_n)$ satisfy \Cref{assumption:sp}.
Then there exists a nonnegative minimizer $u \in \mathcal{D}^{\vec{s},\vec{p}}(\mathbb{R}^n)$ of
\begin{equation} \label{eq:min_prob}
S:= \inf_{u \in \mathcal{D}^{\vec{s},\vec{p}}(\mathbb{R}^n), \|u\|_{L^{p^\ast}(\mathbb{R}^n)}=1} \sum_{i=1}^n \frac{1}{p_i} \Vert D^{s_i}_{p_i}u\Vert_{L^{p_i}(\R^n)}^{p_i}.
\end{equation}
\end{theorem}
A consequence of the foregoing result is the existence of a nontrivial, nonnegative solution to the corresponding critical problem.
\begin{corollary}
\label{cor:existence}
Assume $\vec{s}=(s_1,\dots,s_n)$ and $\vec{p}=(p_1,\dots,p_n)$ satisfy \Cref{assumption:sp}.
Then there exists a nontrivial nonnegative solution $u \in \mathcal{D}^{\vec{s},\vec{p}}(\mathbb{R}^n)$ of
\begin{equation} \label{eq:main_eq}
\sum_{i=1}^n (-\partial_{ii})^{s_i}_{p_i} u = |u|^{p^{\ast}-2}u \quad \text{in}~ \mathbb{R}^n.
\end{equation}
\end{corollary}

We first comment on known results from the literature before we explain the novelty of our results. For $s\in(0,1)$ and $p\in (1,\infty)$, we define the
normalized Gagliardo–Slobodecki\u{ı} seminorm as
\[ [u]_{W^{s,p}(\R^n)} := \left(s(1-s) \int_{\R^n}\int_{\R^n} \frac{|u(x)-u(y)|^p}{|x-y|^{n+sp}}\, \d x \, \d y\right)^{1/p} \]
and for $s=1$ and $p\in(1,\infty)$, we define
\[ [u]_{W^{1,p}(\R^n)} := \|\nabla u\|_{L^p(\R^n)}. \]
The homogeneous Sobolev space is then given as the completion of $C_c^{\infty}(\R^n)$ with respect to the $W^{s,p}(\R^n)$-seminorm. In the subconformal case $sp<n$, it can be represented as the function space
 \begin{equation*}
\mathcal{D}^{s, p}(\mathbb{R}^n) = \lbrace u \in L^{p^\ast}(\mathbb{R}^n): [u]_{W^{s, p}(\mathbb{R}^n)} < \infty \rbrace,
\end{equation*}
where $p^\ast = np/(n-sp)$. For details on homogeneous Sobolev spaces, we refer the reader to \cite{BGV} and the references therein.

The existence of a minimizer of the Sobolev quotient
\begin{equation} \label{eq:Sobolev_quotient}
\inf_{u \in \mathcal{D}^{s, p}(\mathbb{R}^n) \setminus \lbrace 0 \rbrace} \frac{[u]_{W^{s, p}(\mathbb{R}^n)}^p}{\|u\|_{L^{p^\ast}(\mathbb{R}^n)}^p}
\end{equation}
is an important problem in the analysis of variational problems. One reason for this is that minimizers of \eqref{eq:Sobolev_quotient} satisfy the quasilinear equation involving the critical exponent
\begin{equation} \label{eq:semilinear_eq}
(-\Delta)^s_p u = |u|^{p^\ast-2}u \quad \text{in}~ \mathbb{R}^n.
\end{equation}
For $s=1$ and $p=2$, this equation \eqref{eq:semilinear_eq} is related to the Yamabe problem, which addresses a question on the scalar curvature of Riemannian manifolds. 
For $s=1$, the problem of existence of minimizers to \eqref{eq:Sobolev_quotient} in the case $p\in(1,n)$ was completely answered in \cite{Aub76c,Tal76}.
Not only the existence of minimizers to \eqref{eq:Sobolev_quotient} is proved, but also their explicit form is shown to be (up to translation and scaling) 
\begin{equation*}
u(x) = \left(1+|x|^{\frac{p}{p-1}}\right)^{\frac{p-n}{p}}.
\end{equation*}
Let us now comment on results in the fractional case $s\in(0,1)$. 
In the linear case $p=2$, the authors in \cite{CT04} obtain sharp constants for Sobolev inequalities. They prove the existence of minimizers and prove that 
they are (again up to translation and scaling) of the form 
\begin{equation} \label{eq:minimizernonlocal}
u(x) = \left(1+|x|^{\frac{p}{p-1}}\right)^{\frac{sp-n}{p}}.
\end{equation}
In the case $p\in(1,n/s)$, the existence of minimizers is proved in \cite{BMS16}.
Furthermore, it is conjectured that minimizers are of the form \eqref{eq:minimizernonlocal}, see \cite[Equation (1.9)]{BMS16}. The authors provide the asymptotic behavior of minimizers, but their explicit form remains still open. 

The aim of this paper is to show the existence of nonnegative minimizers to \eqref{eq:min_prob} and as a consequence prove the existence of nontrivial nonnegative solutions to the corresponding critical problem \eqref{eq:main_eq}. 
The results in this paper can be seen as a combination of the results in \cite{EHR07} and \cite{BSS18}. Our approach is based on an appropriate adjustment of the concentration-compactness principle. Let us briefly discuss the results in \cite{EHR07} and \cite{BSS18} before we address results concerning the concentration-compactness principle.
In \cite{EHR07}, the authors study the anisotropic $p$-Laplacian \eqref{eq:locaniso}. They prove that nonnegative minimizers of
\begin{equation*}
\inf_{u \in \mathcal{D}^{1,\vec{p}}(\mathbb{R}^n), \|u\|_{L^{q}(\mathbb{R}^n)}=1} \sum_{i=1}^n \frac{1}{p_i} \|D_i u\|_{L^{p_i}(\R^n)}^{p_i},
\end{equation*}
exist, where $\vec{p} = (p_1, \dots, p_n)$ and $q$ are given as in \eqref{eq:p}, and
\begin{equation*}
\mathcal{D}^{1,\vec{p}}(\mathbb{R}^n) = \lbrace u \in L^{q}(\mathbb{R}^n): |D_i u| \in L^{p_i}(\R^n) \rbrace.
\end{equation*}
The main tool in \cite{EHR07} is an adaption of the concentration-compactness principle. The existence of a nonnegative minimizer implies the existence of a nonnegative solution to the corresponding anisotropic critical problem. \\
In \cite{BSS18}, the authors extend the concentration-compactness principle for the fractional \\
$p$-Laplacian in unbounded domains. Using the concentration-compactness principle, they provide sufficient conditions for the existence of a nontrivial solution to the generalized fractional Brezis--Nirenberg problem.\\ Combining ideas from the anisotropic local case as in \cite{EHR07} and from the nonlocal case as in \cite{BSS18} allows us to prove the concentration-compactness principle for the anisotropic nonlocal operator in \eqref{eq:critintro} and the existence of solutions to the corresponding critical problem.

\begin{remark}
We would like to emphasize that our method of proof --- up to straightforward modifications --- carries over to certain anisotropic operators with a possibly partially orthotropic structure, as for example
\begin{equation}
\label{eq:genoperatorex}
\left( - \partial_{11} - \partial_{22}\right)^{s}_{p} + (-\partial_{33})^{t}_{q},
\end{equation}
where $n = 3$, $p,q > 1$, $s,t \in (0,1)$. In case $\vec{s} = (s,s,t), \vec{p} = (p,p,q)$ satisfy \Cref{assumption:sp}, one can prove analogues of  \Cref{thm:sobolev}, \autoref{thm:existence} and \Cref{cor:existence} for the operator in \eqref{eq:genoperatorex}.
\end{remark}

We would like to mention that the classification of nonnegative minimizers \eqref{eq:min_prob} remains an open problem. To the best of our knowledge, even in the anisotropic local case such as \eqref{eq:locaniso}, the shape of minimizers is not known. We do not even know whether such classification is possible in the anisotropic case.  

Let us now refer to results on the concentration-compactness principle in the literature.
As already mentioned, the concentration-compactness principle (CCP) has been introduced by P.-L. Lions in a series of papers. In \cite{Lions3} the CCP was introduced for bounded domains. It has been extended to unbounded domains by Chabrowski in \cite{C95} (see also \cite{BCS95,BTW96}).
For the fractional $p$-Laplacian on bounded domains, the CCP was established in the linear case $p=2$ in \cite{Pala14} and for general $1<p<n/s$ in \cite{MoscSqua16}. The CCP for the fractional $p$-Laplacian was extended to unbounded domains in \cite{BSS18}.

We briefly refer to further results addressing the CCP. In \cite{PQR21}, the authors obtain a CCP in the linear case $p=2$ for a class of stable processes in $\R^n$.
Concerning the $p$-Laplacian resp. the fractional $p$-Laplacian, we refer the reader to \cite{GP91, BMS16}. For results with regard to the fractional $(p,q)$-Laplacian, see \cite{BM19,AAI19,Amb20} and the references therein. The CCP for local operators with variable exponents is studied in \cite{AB13,FBS10,Fu09,FZ10,HS16,HKS19} and the nonlocal case with variable exponents can be found in \cite{HK21}.

The article is organized as follows. In \Cref{sec:Sob}, we prove an anisotropic fractional Sobolev inequality and a compact embedding theorem. 
\Cref{sec:CCP} establishes the CCP and is divided into three subsections. In \Cref{ssec:Decay}, we study the decay of cutoff functions, while \Cref{ssec:reverse} contains the proof of reverse Hölder inequalities. Finally, in \Cref{ssec:main} we prove the main results of this paper.

\section{Sobolev inequality and compact embedding}\label{sec:Sob}
In this section, we prove an anisotropic fractional Sobolev inequality and a compact embedding theorem. 
Even as a self-standing result the anisotropic fractional Sobolev inequality is very interesting. On the one hand, it allows for different orders of integrability $p_1,\dots,p_n>1$ and at the same time also different orders of differentiability $s_1,\dots,s_n\in(0,1)$. Hence, this Sobolev inequality covers plenty of anisotropic nonlocal operators, such as $-(-\partial_{11})^{s_1} - \dots - (-\partial_{nn})^{s_n}$ or fractional orthotropic $p$-Laplacians (see \cite{CK}). 
By comparability of energy forms, one can deduce the Sobolev inequality also for isotropic operators such as the fractional Laplacian $-(-\Delta)^s$. On the other hand, the Sobolev inequality is robust in the sense that the appearing constant depends only on a lower bound of $s_1,\cdots,s_n$. This allows us to recover local isotropic and anisotropic Sobolev inequalities such as Troisi’s inequality, see \eqref{eq:Troisi}. Therefore, \Cref{thm:sobolev} not only extends \eqref{eq:Troisi} to the nonlocal setting, but also provides another proof of it.\\
The compact embedding $\Dsp \subset\subset L^{p}_{\mathrm{loc}}(\R^n)$ for every $1 \le p < p^{\ast}$ is a consequence of the anisotropic  fractional Sobolev inequality and direct computations.\\

The proof of \Cref{thm:sobolev} is based on the boundedness of maximal and sharp maximal operators. For this purpose, we define a class of rectangles taking the anisotropy into account. For $i=1,\dots, n$, set
\begin{equation} \label{eq:m_i}
m_i = \frac{s_{\max}}{s_i} \left( \frac{1}{p_i} - \frac{1}{p^{\ast}} \right) \left( \frac{1}{p_{\max}} - \frac{1}{p^{\ast}} \right)^{-1}.
\end{equation}
Note that $m_i \geq 1$ by the assumption $p_{\max} < p^{\ast}$. For $x \in \mathbb{R}^n$ and $\rho > 0$, we define
\begin{equation} \label{eq:M}
M_{\rho}(x) = \bigtimes_{i=1}^n \left(x_i-\rho^{m_i},x_i+\rho^{m_i} \right).
\end{equation}
The sets $M_{\rho}(x)$ are balls with respect to the metric $d_M$ defined by
\begin{equation*}
d_M(x,y) = \sup_{i = 1,\dots,n} \vert x_i-y_i\vert^{1/m_i}, \quad x,y \in \R^n.
\end{equation*}
Note that $d_M$ is a metric since $m_i \ge 1$. An important ingredient in the proof of \Cref{thm:sobolev}, is the fact that $(\R^n, d_M)$ equipped with the Lebesgue measure is a doubling metric measure space. Since
\begin{equation*}
|M_{2\rho}| = 2^n (2\rho)^{\sum_{i=1}^n m_i} = 2^{\sum_{i=1}^n m_i} |M_{\rho}| = 2^{s_{\max}/(\frac{1}{p_{\max}}-\frac{1}{p^{\ast}})}|M_{\rho}| \leq 2^{\frac{(p^{\ast})^2}{p^{\ast}-p_{\max}}}|M_{\rho}|,
\end{equation*}
the doubling constant depends on $p^{\ast}$ and $p^{\ast}-p_{\max}$.

Let us recall the Hardy–Littlewood maximal function and sharp maximal function: for $u \in L^1_{\mathrm{loc}}(\mathbb{R}^n)$,
\begin{equation*}
{\bf M}u(x) = \sup_{\rho > 0} \fint_{M_\rho(x)} u(y) \,\d y \quad\text{and}\quad {\bf M}^{\sharp}u(x) = \sup_{\rho>0} \fint_{M_\rho(x)} |u(y) - (u)_{M_\rho(x)}| \,\d y,
\end{equation*}
where $(u)_{M_{\rho}(x)} = \fint_{M_{\rho}(x)} u(z)\,\d z$. It is well known that the maximal inequality and sharp maximal inequality hold true on doubling metric spaces, see \cite{Hein01,GrafakosMF,CK}.

\begin{proof} [Proof of \Cref{thm:sobolev}]
Let $u \in \mathcal{D}^{\vec{s},\vec{p}}(\mathbb{R}^n)$, then the maximal function ${\bf M}u$ and sharp maximal function ${\bf M}^{\sharp}u$ are well defined. Following the proof of \cite[Theorem 2.4]{CK}, one obtains
\begin{equation} \label{eq:F_i}
\fint_{M_\rho(x)} |u(y) - (u)_{M_\rho(x)}|\, \d y \leq \sum_{i=1}^n \rho^{s_i m_i} \fint_{M_{\rho}(x)} F_i(y) \,\mathrm{d}y,
\end{equation}
where the function $F_i$ is defined by
\begin{equation*}
F_i(y) := \sup_{\rho > 0} \left( 2\rho^{-s_im_i} \fint_{-2\rho^{m_i}}^{2\rho^{m_i}} |u(y) - u(y+he_i)| \,\mathrm{d}h \right).
\end{equation*}
By H\"older's inequality, we have
\begin{equation} \label{eq:F_i_Holder}
\begin{split}
\left( \fint_{M_\rho(x)} F_i(y) \,\mathrm{d}y \right)^{p^{\ast}} 
&\leq \left( \fint_{M_{\rho}(x)} F_i^{p_i}(y) \,\mathrm{d}y \right)^{\frac{{p^{\ast}}-p_i}{p_i}} \left( \fint_{M_{\rho}(x)} F_i(y) \,\mathrm{d}y \right)^{p_i} \\
&\leq |M_{\rho}|^{-\frac{{p^{\ast}}-p_i}{p_i}} \|F_i\|_{L^{p_i}(\mathbb{R}^n)}^{{p^{\ast}}-p_i} \left( \fint_{M_{\rho}(x)} F_i(y) \,\mathrm{d}y \right)^{p_i}.
\end{split}
\end{equation}
Thus, it follows from \eqref{eq:F_i} and \eqref{eq:F_i_Holder} that
\begin{equation*}
\begin{split}
&\left( \fint_{M_\rho(x)} |u(y) - (u)_{M_\rho(x)}|\,\mathrm{d}y \right)^{p^{\ast}} \\
&\leq \sum_{i=1}^n n^{{p^{\ast}}-1} \rho^{{p^{\ast}}s_im_i} \left( \fint_{M_{\rho}(x)} F_i(y) \,\mathrm{d}y \right)^{p^{\ast}} \\
&\leq \sum_{i=1}^n n^{{p^{\ast}}-1} \rho^{{p^{\ast}}s_i m_i} |M_{\rho}|^{-\frac{{p^{\ast}}-p_i}{p_i}} \|F_i\|_{L^{p_i}(\mathbb{R}^n)}^{{p^{\ast}}-p_i} \left( \fint_{M_{\rho}(x)} F_i(y) \,\mathrm{d}y \right)^{p_i}.
\end{split}
\end{equation*}
Note that for each $i$,
\begin{equation*}
p^{\ast} s_im_i - \frac{p^{\ast}-p_i}{p_i} \sum_{j=1}^n m_j = 0.
\end{equation*}
Therefore, we arrive at
\begin{equation*}
\left( \fint_{M_\rho(x)} |u(y) - (u)_{M_\rho(x)}|\,\mathrm{d}y \right)^{p^{\ast}} \leq \sum_{i=1}^n \frac{n^{{p^{\ast}}-1}}{2^{n\frac{{p^{\ast}}-p_i}{p_i}}} \|F_i\|_{L^{p_i}(\mathbb{R}^n)}^{{p^{\ast}}-p_i} \left( \fint_{M_{\rho}(x)} F_i(y) \,\mathrm{d}y \right)^{p_i}.
\end{equation*}
We take the supremum over $\rho > 0$ and use the estimate $2^{n\frac{p^{\ast}-p_i}{p_i}} \geq 1$, then
\begin{equation*}
\left( {\bf M}^\sharp u(x) \right)^{p^{\ast}} \leq n^{p^{\ast}-1} \sum_{i=1}^n \|F_i\|_{L^{p_i}(\mathbb{R}^n)}^{{p^{\ast}}-p_i} ({\bf M}F_i(x))^{p_i}.
\end{equation*}
By boundedness of the sharp maximal operator in $L^{p^{\ast}}(\R^n)$ and the maximal operator in $L^{p_i}(\R^n)$,
\begin{equation*}
\|u\|_{L^{p^{\ast}}(\mathbb{R}^n)}^{p^{\ast}} \leq C \|{\bf M}^\sharp u\|_{L^{p^{\ast}}(\mathbb{R}^n)}^{p^{\ast}} \leq C \sum_{i=1}^n \|F_i\|_{L^{p_i}(\mathbb{R}^n)}^{{p^{\ast}}-p_i} \|{\bf M}F_i\|_{L^{p_i}(\R^n)}^{p_i} \leq C \sum_{i=1}^n \|F_i\|_{L^{p_i}(\mathbb{R}^n)}^{p^{\ast}}
\end{equation*}
for some $C > 0$. The constant $C$ depends on the doubling constant and the number of intersections of anisotropic dyadic rectangles as in \cite[Appendix B]{CK}. Therefore, $C = C(n, p^{\ast}, p^{\ast}-p_{\max}, s_0)$.

It only remains to show that there is a $C=C(n,p^{\ast}, p^{\ast}-p_{\max}, s_0)>0$ such that
\begin{equation*}
\|F_i\|_{L^{p_i}(\mathbb{R}^n)}^{p_i} \leq C s_i(1-s_i) \int_{\mathbb{R}^n} \int_{\mathbb{R}} \frac{|u(x)-u(x+he_i)|^{p_i}}{|h|^{1+s_ip_i}} \,\mathrm{d}h \,\mathrm{d}x
\end{equation*}
for each $i=1,\dots, n$, which follows exactly as in \cite[Lemma 2.3 and Theorem 2.4]{CK}.
\end{proof}

\begin{theorem}[Compact Embedding] \label{thm:compemb}
Let $s_1,\dots,s_n\in[s_0,1)$ for some $s_0\in(0,1)$ and $p_1,\dots,p_n>1$ be such that \Cref{assumption:sp} holds. Then
$\Dsp \subset\subset L^{p}_{\mathrm{loc}}(\R^n)$ for every $1 \le p < p^{\ast}$.
\end{theorem}

\begin{proof}
We follow the proof in \cite[Theorem 4.54]{DD12}.\\
Let $\mathcal{B} \subset \Dsp$ be bounded and $\Omega \subset \R^n$ be an open, bounded set. By \Cref{thm:sobolev} we know that $\mathcal{B}$ is bounded in $L^{p^{\ast}}(\R^n)$. Therefore, by a standard H\"older interpolation argument, it is sufficient to prove that $\mathcal{B}$ is precompact in $L^1(\Omega)$ in order to deduce the desired result. We will show that there is $C > 0$ such that for every $r \in \R$, $i \in \{1,\dots,n\}$, $u \in \mathcal{B}$ it holds:
\begin{equation}
\label{eq:compembHelp1}
\int_{\Omega_r} \vert u(x+r e_i) - u(x) \vert\,  \d x \le C \vert r \vert^{s_i} \vert \Omega \vert^{(p_i-1)/p_i} \Vert D^{s_i}_{p_i} u \Vert_{L^{p_i}(\R^n)},
\end{equation}
where $\Omega_r = \{ x \in \R^n : \dist(x,\partial \Omega) > \vert r\vert \}$.
By the triangle inequality, we have
\begin{equation*}
\begin{split}
&2\vert r\vert \int_{\Omega_r} \vert u(x+r e_i) - u(x) \vert \, \d x \\
& \le \int_{\Omega_r} \int_{(-\vert r\vert,\vert r \vert)}\vert u(x+r e_i) - u(x+t e_i) \vert + \vert u(x+te_i) - u(x) \vert \, \d t \, \d x  \\
&= \int_{\Omega_r} \int_{(-\vert r\vert,\vert r\vert)}\frac{\vert u(x+r e_i) - u(x+t e_i) \vert}{\vert t-r\vert^{\frac{1+s_ip_i}{p_i}}} \vert t-r\vert^{\frac{1+s_ip_i}{p_i}} \, \d t \, \d x \\
& \qquad\qquad+ \int_{\Omega_r} \int_{(-\vert r \vert,\vert r\vert)} \frac{\vert u(x+te_i) - u(x) \vert}{\vert t \vert^{\frac{1+s_ip_i}{p_i}}} \vert t \vert^{\frac{1+s_ip_i}{p_i}}\, \d t \, \d x.
\end{split}
\end{equation*}
Both expressions on the right-hand side can be estimated from above in a similar way. We demonstrate how to treat the second expression. 
By H\"older's inequality
\begin{equation*}
\begin{split}
\int_{\Omega_r} &\int_{(-\vert r \vert,\vert r\vert)} \frac{\vert u(x+te_i) - u(x) \vert}{\vert t \vert^{\frac{1+s_ip_i}{p_i}}} \vert t \vert^{\frac{1+s_ip_i}{p_i}}\, \d t \, \d x\\
&\le \left(\int_{\Omega_r} \int_{(-\vert r\vert,\vert r\vert)} \frac{\vert u(x+te_i) - u(x) \vert^{p_i}}{\vert t \vert^{1+s_ip_i}} \, \d t \, \d x \right)^{1/p_i} \left( \int_{\Omega_r} \int_{(-\vert r\vert,\vert r\vert)} \vert t \vert^{\frac{1+s_ip_i}{p_i-1}} \, \d t \,  \d x \right)^{\frac{p_i-1}{p_i}}\\
&\le C\vert r \vert^{1+s_i} \vert \Omega \vert^{1-1/p_i} \Vert D^{s_i}_{p_i} u \Vert_{L^{p_i}(\R^n)} 
\end{split}
\end{equation*}
for some $C>0$. This proves \eqref{eq:compembHelp1}. Now let $h \in B_1(0)$.
Then
\begin{equation*}
\begin{split}
& \int_{\Omega_{2|h|}} \vert u(x+h) - u(x) \vert \, \d x \\
&\leq \sum_{i=1}^n \int_{\Omega_{2|h|}} \vert u(x+h_1e_1+\cdots+h_i e_i) - u(x+h_1e_1+\cdots+h_{i-1}e_{i-1}) \vert \, \d x \\
&= \sum_{i=1}^n \int_{h_1e_1+\cdots+h_{i-1}e_{i-1}+\Omega_{2|h|}} \vert u(x+h_i e_i) - u(x) \vert \, \d x.
\end{split}
\end{equation*}
Since $h_1e_1+\dots+h_{i-1}e_{i-1}+\Omega_{2|h|} \subset \Omega_{h_i}$ for every $i \in \{1,\dots,n\}$, we have
\begin{equation*}
\begin{split}
\int_{\Omega_{2|h|}} \vert u(x+h) - u(x) \vert \, \d x
&\leq \sum_{i=1}^n \int_{\Omega_{h_i}} \vert u(x+h_i e_i) - u(x) \vert \, \d x \\
&\le C \vert h \vert^{s_{0}} \max \lbrace 1, \vert \Omega \vert^{1-1/p_{\max}} \rbrace \Vert u \Vert_{\Dsp}.
\end{split}
\end{equation*}
for some $C>0$. By boundedness of $\mathcal{B}$ in $\Dsp$, the right-hand side vanishes uniformly as $h \searrow 0$. Since $\mathcal{B}$ is also bounded in $L^1(\Omega)$, we can deduce precompactness of $\mathcal{B}$ in $L^1(\Omega)$ from the Riesz–Fr\'echet–Kolmogorov theorem (c.f. \cite[Theorem 1.95]{DD12}).
\end{proof}

\section{Concentration-compactness principle} \label{sec:CCP}

In this section we present the concentration-compactness principle for solving the minimization problem \eqref{eq:min_prob}. 
Throughout this section, $\vec{s}$ and $\vec{p}$ always satisfy \Cref{assumption:sp}.

Recall that
\begin{equation*}
S= \inf_{u \in \mathcal{D}^{\vec{s},\vec{p}}(\mathbb{R}^n), \|u\|_{L^{p^\ast}(\mathbb{R}^n)}=1} \sum_{i=1}^n \frac{1}{p_i} \Vert D^{s_i}_{p_i}u\Vert_{L^{p_i}(\R^n)}^{p_i},
\end{equation*}
which is positive by \Cref{thm:sobolev}. Furthermore, by definition of $S$, we have the following lemma.
\begin{lemma}
\label{lem:SobCor}
For every $u \in \Dsp$ with $\Vert u \Vert_{L^{p^{\ast}}(\R^n)} \le 1$, it holds that
\begin{equation*}
S \Vert u \Vert_{L^{p^{\ast}}(\R^n)}^{p_{\max}} \le \sum_{i=1}^n \frac{1}{p_i} \Vert D_{p_i}^{s_i} u \Vert_{L^{p_i}(\R^n)}^{p_i}.
\end{equation*}
\end{lemma}

\begin{proof}
Let $u \in \Dsp$ be such that $u \neq 0$ and $\Vert u \Vert_{L^{p^{\ast}}(\R^n)} \le 1$. Let $v = u/\|u\|_{L^{p^{\ast}}(\R^n)}$, then $\|v\|_{L^{p^{\ast}}(\R^n)} = 1$. Thus, by the definition of $S$,
\begin{equation*}
S \leq \sum_{i=1}^n \frac{1}{p_i} \Vert D_{p_i}^{s_i} v \Vert_{L^{p_i}(\R^n)}^{p_i} = \sum_{i=1}^n \frac{1}{p_i} \frac{\Vert D_{p_i}^{s_i} u \Vert_{L^{p_i}(\R^n)}^{p_i}}{\|u\|_{L^{p^{\ast}}(\R^n)}^{p_i}} \leq \sum_{i=1}^n \frac{1}{p_i} \frac{\Vert D_{p_i}^{s_i} u \Vert_{L^{p_i}(\R^n)}^{p_i}}{\|u\|_{L^{p^{\ast}}(\R^n)}^{p_{\max}}}.
\end{equation*}
\end{proof}
The following scaling property is a consequence of a straightforward computation.
\begin{lemma} \label{lem:scaling}
For $i=1,\dots,n$, let $m_i$ be the exponent defined by \eqref{eq:m_i}. For every $y \in \mathbb{R}^n$, $\rho > 0$, and $u \in \mathcal{D}^{\vec{s},\vec{p}}(\mathbb{R}^n)$, the function $v$ defined by
\begin{equation} \label{eq:v}
v(x) = \rho^{s_{\max}p_{\max}/(p^{\ast}-p_{\max})} u(\rho^{m_1}x_1+y_1, \cdots, \rho^{m_n}x_n+y_n)
\end{equation}
satisfies $\|v\|_{\mathcal{D}^{\vec{s},\vec{p}}(\mathbb{R}^n)} = \|u\|_{\mathcal{D}^{\vec{s},\vec{p}}(\mathbb{R}^n)}$ and $\Vert v \Vert_{L^{p^{\ast}}(\R^n)} = \Vert u \Vert_{L^{p^{\ast}}(\R^n)}$.
\end{lemma}

To prove the existence of a minimizer of $S$, let us take a minimizing sequence $(u_k) \subset \Dsp$ for $S$ satisfying $\|u_k\|_{L^{p^{\ast}}(\mathbb{R}^n)} = 1$ for every $k \in \N$. We define the L\'evy concentration function
\begin{equation*}
Q_k(\rho) = \sup_{y \in \mathbb{R}^n} \int_{M_{\rho}(y)} |u_k|^{p^{\ast}} \, \d x, \quad \rho > 0,
\end{equation*}
where $M_{\rho}(y)$ is the rectangle defined by \eqref{eq:M}. For each $k \in \mathbb{N}$, there exists $\rho_k > 0$ such that $Q_k(\rho_k)=1/2$. Moreover, there exists $y_k \in \mathbb{R}^n$ such that
\begin{equation*}
Q_k(\rho_k) = \int_{M_{\rho_k}(y_k)} |u_k|^{p^{\ast}} \, \d x = \frac{1}{2}.
\end{equation*}
By \Cref{lem:scaling}, the function $v_k$ defined by \eqref{eq:v} with $\rho=\rho_k$ satisfies
\begin{equation} \label{eq:sup}
\int_{M_1(0)} |v_k|^{p^{\ast}} \,\mathrm{d}x = \frac{1}{2} = \sup_{y \in \mathbb{R}^n} \int_{M_1(y)} |v_k|^{p^{\ast}} \,\mathrm{d}x.
\end{equation}
Furthermore, the sequence $(v_k)$ is also a minimizing sequence for $S$. By the Banach–Alaoglu theorem and \Cref{thm:sobolev}, we have
\begin{equation} \label{eq:conv_Dsp}
v_k \rightharpoonup v ~ \text{in}~ \Dsp ~\text{and}~ L^{p^{\ast}}(\mathbb{R}^n)
\end{equation}
for some $v \in \Dsp$, up to subsequences. Moreover, we have by  \autoref{thm:compemb} that 
\begin{equation} \label{eq:conv_ptwise}
v_k \to v ~ \text{a.e. in}~ \mathbb{R}^n
\end{equation}
up to subsequences. By the lower semicontinuity of norms $\Vert\cdot\Vert_{\Dsp}$ and $\|\cdot\|_{L^{p^{\ast}}(\mathbb{R}^n)}$, we have
\begin{equation} \label{eq:lower_semicont}
\Vert v\Vert_{\Dsp} \leq \liminf_{k\to\infty} \Vert v_k\Vert_{\Dsp} = S \quad\text{and}\quad \|v\|_{L^{p^{\ast}}(\mathbb{R}^n)} \leq \liminf_{k\to\infty} \|v_k\|_{L^{p^{\ast}}(\mathbb{R}^n)} = 1.
\end{equation}
Therefore, it is enough to prove the following theorem to find a minimizer for $S$.
\begin{theorem} \label{thm:unit_norm}
Let $v:\R^n\to\R$ be the defined as above. Then
\[\|v\|_{L^{p^{\ast}}(\mathbb{R}^n)} = 1.\]
\end{theorem}
This is the main theorem of this section. We will provide the proof at the beginning of \Cref{ssec:main}.

\subsection{Decay of cutoff functions}\label{ssec:Decay}

Let us make a few comments on the significance of cutoff functions in this article. Unlike in the local case, the support of a function is not preserved by the application of a nonlocal operator. This basic fact causes some problems since the compact embedding $\Dsp \subset\subset L^{p}_{\mathrm{loc}}(\R^n)$, $1 \le p < p^{\ast}$, (c.f. \Cref{thm:compemb}) cannot be used directly. In \cite{BSS18}, the authors solve this issue by proving a compact embedding into weighted spaces of the form $L^{p}(\vert D^s \phi(x) \vert^p \d x)$, where $\phi \in C_c^{\infty}(\R^n)$ is a cutoff function. The main observation is that the fractional $(s,p)$-gradient of a compactly supported function decays fast enough for such result to hold true, despite having full support.

The following lemmas discuss the adaption of this argument to the class of anisotropic nonlocal operators considered in this article.

For $x \in \R^n$, $\rho > 0$, and $i\in\{1,\dots,n\}$, we introduce the notation
\begin{equation*}
M_{\rho}^{(i)}(x) = \bigtimes_{j=1}^{i-1} \left(x_j-\rho^{m_j},x_j+\rho^{m_j} \right) \times \R \times \bigtimes_{j=i+1}^{n} \left(x_j-\rho^{m_j},x_j+\rho^{m_j} \right).
\end{equation*}

\begin{lemma}
\label{lem:derdecay}
Let $\phi \in C_c^{\infty}(\R^n)$ be such that $\supp(\phi) \subset M_1(0)$. Then
\begin{equation}
\label{eq:derdecay}
\vert D^{s_i}_{p_i}\phi(x)\vert^{p_i} \le C \min(1,\vert x_i \vert^{-1-s_ip_i}), \quad i = 1,\dots,n,
\end{equation}
where $C = C(n,s_i,p_i,\Vert\phi\Vert_{\infty},\Vert \partial_i \phi\Vert_{\infty})>0$. Moreover, $\supp(\vert D^{s_i}_{p_i}\phi\vert) \subset M_1^{(i)}(0)$.
\end{lemma}

\begin{proof}
We first observe that for given $j \in \{1,\dots,n\}$ and $x \in \R^n$ with $\vert x_j \vert > 1$, it holds that $\phi(x+e_ih) = 0$ for every $h \in \R$ and $i \neq j$. It becomes evident that $\supp(\vert D^{s_i}_{p_i}\phi\vert) \subset M_1^{(i)}(0)$ for every $i \in \{1,\dots,n\}$.\\
For $x \in \R^n$ we compute
\begin{equation*}
s_i(1-s_i)\int_{\mathbb{R}} \frac{|\phi(x)-\phi(x+e_ih)|^{p_i}}{|h|^{1+s_i p_i}} \,\d h \le C(p_i,n) \left(\Vert \phi \Vert_{\infty}^{p_i} + \Vert \partial_i \phi\Vert_{\infty}^{p_i}\right).
\end{equation*}
If $\vert x_i \vert > 2$ and $\vert x_i + h\vert < 1$, then $\vert h \vert > \vert x_i\vert /2$. Therefore, we estimate for $x \in \R^n$ with $\vert x_i \vert > 2$:
\begin{equation*}
\begin{split}
s_i(1-s_i)\int_{\mathbb{R}} \frac{|\phi(x)-\phi(x+e_ih)|^{p_i}}{|h|^{1+s_i p_i}} \,\d h 
&= s_i(1-s_i)\int_{\vert x_i + h \vert < 1} \frac{\vert\phi(x + e_ih)\vert^{p_i}}{\vert h \vert^{1+s_ip_i}} \, \d h \\
&\le C \frac{\Vert \phi\Vert_{\infty}^{p_i}}{\vert x_i \vert^{1+s_ip_i}}.
\end{split}
\end{equation*}
Assertion \eqref{eq:derdecay} follows by combining the aforementioned estimates.
\end{proof}

Given $\phi \in C_c^{\infty}(\R^n)$, $\eps > 0$, and $\xi \in \R^n$, we introduce the scaled cutoff function $\phi_{\eps,\xi}$ defined by 
\begin{equation}\label{def:scaledcutoff}
\phi_{\eps,\xi}(x) = \phi(x^{(\eps,\xi)}),
\end{equation}
where $x^{(\eps,\xi)} := ((x_1-\xi_1) \eps^{-m_1},\dots,(x_n-\xi_n) \eps^{-m_n})$ and $m_i$ is given by \eqref{eq:m_i}. 

One easily computes:
\begin{equation*}
\vert D^{s_i}_{p_i}\phi_{\eps,0}(x)\vert^{p_i} = \eps^{-m_is_ip_i}\vert D^{s_i}_{p_i}\phi(x^{(\eps,0)})\vert^{p_i}.
\end{equation*}

Consequently, we have the following corollary of \Cref{lem:derdecay}:

\begin{corollary}
\label{cor:derdecay}
Let $\phi \in C_c^{\infty}(\R^n)$ be such that $\supp(\phi) \subset M_1(0)$. Let $\eps > 0$ and $\xi \in \R^n$. Then
\begin{equation*}
\vert D^{s_i}_{p_i}\phi_{\eps,\xi}(x)\vert^{p_i} \le C \min(\eps^{-m_is_ip_i},\eps^{m_i}\vert x_i-\xi_i \vert^{-1-s_ip_i}), \quad i = 1,\dots,n,
\end{equation*}
where $C = C(n,s_i,p_i,\Vert\phi\Vert_{\infty},\Vert \partial_i \phi\Vert_{\infty})>0$.
Moreover, $\supp(\vert D^{s_i}_{p_i}\phi_{\eps,\xi}\vert) \subset M_{\eps}^{(i)}(\xi)$.
\end{corollary}

The following lemma makes use of the compact embedding \Cref{thm:compemb} and the decay estimates. 
It follows in the spirit of \cite[Lemma 2.4]{BSS18} and can be interpreted as a compact embedding of $\Dsp$ into the weighted space $L^{p_i}(\vert D^{s_i}_{p_i}\phi_{\eps,\xi}\vert^{p_i} \,\d x)$ for every $i \in \{1,\dots,n\}$. Let us recall the function $v_k$ constructed in \Cref{sec:CCP} and define \[w_k := v_k - v.\]

\begin{lemma}
\label{lem:klimit}
Let $\phi \in C_c^{\infty}(\R^n)$ be such that $\supp(\phi) \subset M_1(0)$. Let $\xi \in \R^n$ and $\eps > 0$. 
Then for every $i \in \{1,\dots,n\}$,
\begin{equation*}
\int_{\R^n} \vert D_{p_i}^{s_i} (\phi_{\eps,\xi})w_k\vert^{p_i} \, \d x \to 0,
\end{equation*}
as $k \to \infty$.
\end{lemma}

\begin{proof}
 For $R >0$ to be chosen later, we decompose
\begin{equation*}
\int_{\R^n} \vert D^{s_i}_{p_i}(\phi_{\eps,\xi})w_k \vert^{p_i} \, \d x = \int_{\vert x_i - \xi_i \vert < R} \vert D^{s_i}_{p_i}(\phi_{\eps,\xi})w_k \vert^{p_i} \, \d x + \int_{\vert x_i - \xi_i \vert > R} \vert D^{s_i}_{p_i}(\phi_{\eps,\xi})w_k \vert^{p_i} \, \d x =: I + II.
\end{equation*}
In order to bound $II$, we apply H\"older's inequality, \Cref{lem:scaling}, and \Cref{cor:derdecay}:
\begin{equation*}
\begin{split}
II &\le C \left(\Vert v_k \Vert_{L^{p^{\ast}}(\R^n)}^{p_i/p^{\ast}} + \Vert v \Vert_{L^{p^{\ast}}(\R^n)}^{p_i/p^{\ast}} \right)\left( \int_{\vert x_i -\xi_i\vert > R} \vert D^{s_i}_{p_i}\phi_{\eps,\xi}\vert^{p_i\left(\frac{p^{\ast}}{p^{\ast} - p_i}\right)}\, \d x\right)^{\frac{p^{\ast}-p_i}{p^{\ast}}}\\
&\le C \left( \int_{M_{\eps}^{(i)}(\xi) \cap \{\vert x_i - \xi_i \vert > R\}} \vert x_i- \xi_i\vert^{(-1-s_ip_i)\left(\frac{p^{\ast}}{p^{\ast} - p_i}\right)} \, \d x\right)^{\frac{p^{\ast}-p_i}{p^{\ast}}}\\
&\le C\left(\int_{\vert x_i \vert > R} \vert x_i\vert^{(-1-s_ip_i)\left(\frac{p^{\ast}}{p^{\ast} - p_i}\right)} \, \d x_i\right)^{\frac{p^{\ast}-p_i}{p^{\ast}}}\\
&\le C R^{\gamma},
\end{split}
\end{equation*}
where $C >0$ is independent of $k$, and $\gamma = \left((-1-s_ip_i)\left(\frac{p^{\ast}}{p^{\ast} - p_i}\right)+1\right)\frac{p^{\ast}-p_i}{p^{\ast}} = -p_i(s_i+1/p^{\ast}) < 0$. Thus, $II \to 0$ as $R \to \infty$ uniformly in $k$.\\
For $I$, we observe that the integration takes place on the compact set 
\begin{equation*}
\Omega := M_{\eps}^{(i)}(\xi) \cap \{\vert x_i - \xi_i\vert < R \} \subset \R^n
\end{equation*}
due to \Cref{cor:derdecay}. From \Cref{thm:compemb}, we know that $\Vert w_k \Vert_{L^{p_i}(\Omega)} \to 0$ as $k \to \infty$. Hence, it follows that $I \to 0$ as $k \to \infty$ for any $R > 0$ upon recalling that $\vert D^{s_i}_{p_i}\phi_{\eps,\xi}\vert$ is bounded due to \Cref{cor:derdecay}.
\end{proof}

We end this section by another convergence result for sequences of scaled cutoff functions $(\phi_{\eps,\xi})_{\eps}$ given $\xi \in \R^n$. Let $v$ be the function given in \Cref{thm:unit_norm}.

\begin{lemma}
\label{lem:epslimit}
Let $\phi \in C_c^{\infty}(\R^n)$ be such that $\supp(\phi) \subset M_1(0)$. Let $\xi \in \R^n$ and $\eps > 0$. Then for every $i \in \{1,\dots,n\}$,
\begin{equation*}
\int_{\R^n} \vert D_{p_i}^{s_i} (\phi_{\eps,\xi})v\vert^{p_i} \, \d x \to 0,
\end{equation*}
as $\eps \to 0$.
\end{lemma}

\begin{proof}
For $\eps > 0$ we estimate, using \Cref{cor:derdecay},
\begin{equation*}
\begin{split}
\int_{\R^n} \vert D_{p_i}^{s_i} (\phi_{\eps,\xi})v\vert^{p_i} \, \d x &\le \eps^{-m_is_ip_i} \int_{M_{\eps}(\xi)}  \vert v(x) \vert^{p_i} \, \d x + \eps^{m_i}\int_{M_{\eps}^{(i)}(\xi) \cap \{\vert x_i - \xi_i \vert > \eps^{m_i}\}} \frac{\vert v (x) \vert^{p_i}}{\vert x_i - \xi_i \vert^{1+s_ip_i}} \, \d x\\
&=: I + II.
\end{split}
\end{equation*}
The first summand can be bounded from above by H\"older's inequality:
\begin{equation*}
I \le \eps^{-m_is_ip_i}\left(\int_{M_{\eps}(\xi)} \vert v(x) \vert^{p^{\ast}} \, \d x\right)^{p_i/p^{\ast}} \vert M_{\eps}(\xi)\vert^{\frac{p^{\ast}-p_i}{p^{\ast}}} \le C \Vert v \Vert_{L^{p^{\ast}}(M_{\eps}(\xi))}^{p_i}
\end{equation*}
for some $C>0$, where we used that $m_is_ip_i = \left(1-\frac{p_i}{p^{\ast}}\right)\sum_{j =1}^n m_j$. Since $v \in L^{p^{\ast}}(\R^n)$ it is clear that $I \to 0$ as $\eps \searrow 0$.\\
For $II$, we decompose the domain of integration as follows:
\begin{equation*}
\begin{split}
II &= \sum_{l = 0}^{\infty} \eps^{m_i}\int_{M^{(i)}_{\eps}(\xi) \cap \{(2^{l+1}\eps)^{m_i} >\vert x_i - \xi_i \vert > (2^l\eps)^{m_i}\}} \frac{\vert v (x) \vert^{p_i}}{\vert x_i - \xi_i \vert^{1+s_ip_i}} \, \d x\\
&\le \sum_{l = 0}^{\infty} \eps^{-m_is_ip_i} 2^{-l m_i(1+s_ip_i)} \int_{M_{2^{l+1}\eps}(\xi)} \vert v(x) \vert^{p_i} \, \d x.
\end{split}
\end{equation*}
By H\"older's inequality, we deduce that
\begin{equation*}
II \le c \sum_{l = 0}^{\infty} 2^{-l m_i} \left(\int_{M_{2^{l+1}\eps}(\xi)} \vert v(x) \vert^{p^{\ast}} \, \d x \right)^{p_i /p^{\ast}}.
\end{equation*}
Note that for every $\delta > 0$, we can find $l_0 \in \N$ such that $\sum_{l = l_0+1}^{\infty} 2^{-l m_i} < \delta$. Therefore
\begin{equation}
\label{eq:epslimitHelp1}
II \le \delta \Vert v \Vert_{L^{p^{\ast}}(\R^n)}^{p_i} + C\Vert v \Vert_{L^{p^{\ast}}\left(M_{2^{l_0+1}\eps}(\xi)\right)}^{p_i},
\end{equation}
where $C = C(l_0) > 0$. Since \eqref{eq:epslimitHelp1} holds true for every $\delta > 0$, also $II \to 0$ as $\eps \searrow 0$. This implies the desired result. 
\end{proof}

\subsection{Reverse H\"older inequalities}\label{ssec:reverse}

In this subsection, we prove reverse H\"older inequalities for the ``components'' of limiting measures of the minimizing sequence that spread or concentrate. 
Recall the sequence $(v_k)$ constructed in \Cref{sec:CCP}. We may take a further subsequence of $(v_k)$ so that
\begin{align*}
|D^{s_i}_{p_i} (v_k-v)|^{p_i} \,\mathrm{d}x &\overset{v}{\rightharpoonup} \mu_i,\\
|v_k-v|^{p^{\ast}} \,\mathrm{d}x &\overset{v}{\rightharpoonup} \nu,\\
\sum_{i =1}^n \frac{1}{p_i}\vert D_{p_i}^{s_i} v_k\vert^{p_i}\, \d x &\overset{v}{\rightharpoonup} \widetilde{\mu}
\end{align*}
for some positive bounded Borel measures $\mu_i$, $\nu$, and $\widetilde{\mu}$,  where $\overset{v}{\rightharpoonup}$ denotes vague convergence. 
Moreover, we define $\mu = \sum_{i=1}^n \frac{1}{p_i} \mu_i$. These quantities contain information on the concentration of $(v_k)$.
We also define the limits
\begin{equation*}
\begin{split}
&\mu_{\infty} = \lim_{R \to \infty} \limsup_{k\to\infty} \sum_{i=1}^n \frac{1}{p_i} \int_{\R^n \setminus M_R(0)} |D^{s_i}_{p_i} v_k|^{p_i} \,\mathrm{d}x, \\
&\nu_{\infty} = \lim_{R \to \infty} \limsup_{k\to\infty} \int_{\R^n \setminus M_R(0)} |v_k|^{p^{\ast}} \,\mathrm{d}x,
\end{split}
\end{equation*}
which encode information on the spreading of $(v_k)$. 

In the following, we will make use of the estimate
\begin{equation}\label{eq:minkw}
 \Vert D^{s_i}_{p_i}(f g)\Vert_{L^{p_i}(\R^n)} \le \Vert D^{s_i}_{p_i}(f)g\Vert_{L^{p_i}(\R^n)} + \Vert D^{s_i}_{p_i}(g)f \Vert_{L^{p_i}(\R^n)},
\end{equation}
which holds true for all functions $f,g:\R^n\to\R$, whenever the aforementioned quantities are finite.
This inequality is an immediate consequence of Minkowski's inequality.

Our first result provides reverse H\"older inequalities for $\mu_{\infty}$ and $\nu_{\infty}$.
\begin{lemma}\label{lem:inftyrevHolder} 
For $\mu_{\infty}$ and $\nu_{\infty}$ as above, the following is true:
\[S\nu_{\infty}^{p_{\max}/p^{\ast}} \leq \mu_{\infty}.\]
\end{lemma}

\begin{proof}
Let $\phi \in C_c^{\infty}(\R^n)$ be such that $0 \le \phi \le 1$, $\phi \equiv 1$ in $M_1(0)$ and $\phi \equiv 0$ in $\R^n \setminus M_2(0)$. For $R > 0$, we define $\psi_R := 1 - \phi_{R,0}$, where $\phi_{R,0}$ is defined as in \eqref{def:scaledcutoff}. Then, $\psi_R \equiv 1$ in $\R^n \setminus M_{2R}(0)$ and $\psi_R \equiv 0$ in $M_{R}(0)$. We first prove that
\begin{align}
\label{eq:nuChar}
\nu_{\infty} &= \lim_{R \to \infty} \limsup_{k \to \infty} \int_{\mathbb{R}^n} |w_k|^{p^{\ast}} \psi_R^{p^{\ast}} \,\mathrm{d}x \quad\text{and}\\
\label{eq:muChar}
\mu_{\infty} &= \lim_{R \to \infty} \limsup_{k \to \infty} \sum_{i=1}^n \frac{1}{p_i} \int_{\mathbb{R}^n} |D^{s_i}_{p_i}w_k|^{p_i} \psi_R^{p_i} \,\mathrm{d}x.
\end{align}
For \eqref{eq:nuChar} we observe that 
\begin{equation*}
\int_{\R^n \setminus M_{2R}(0)} \vert v_k \vert^{p^{\ast}} \, \d x \le \int_{\R^n} \vert v_k \vert^{p^{\ast}} \psi_R^{p^{\ast}} \, \d x \le \int_{\R^n \setminus M_{R}(0)} \vert v_k \vert^{p^{\ast}}\, \d x,
\end{equation*}
which implies that
\begin{equation}
\label{eq:nuinftyrep1}
\nu_{\infty} = \lim_{R \to \infty} \limsup_{k \to \infty} \int_{\mathbb{R}^n} |v_k|^{p^{\ast}} \psi_R^{p^{\ast}} \,\mathrm{d}x.
\end{equation}
Thus, the identity \eqref{eq:nuChar} follows now upon the observation that
\begin{equation*}
\lim_{R \to \infty} \int_{\R^n} \vert v \vert^{p^{\ast}} \psi_R^{p^{\ast}} \, \d x = \lim_{R \to \infty} \int_{\R^n \setminus M_R(0)} \vert v \vert^{p^{\ast}} \psi_R^{p^{\ast}} \, \d x = 0,
\end{equation*}
since $v \in L^{p^{\ast}}(\R^n)$. The proof of \eqref{eq:muChar} is analogous.

Let us next prove that $S\nu_{\infty}^{p_{\max}/p^{\ast}} \leq \mu_{\infty}$.
By Brezis–Lieb’s lemma and \eqref{eq:conv_ptwise} we have \[\lim_{k\to\infty}\|w_k\|_{L^{p^{\ast}}(\R^n)}^{p^{\ast}} = \lim_{k\to\infty}\|v_k\|_{L^{p^{\ast}}(\R^n)}^{p^{\ast}}-\|v\|_{L^{p^{\ast}}(\R^n)}^{p^{\ast}}<1.\] It follows from \Cref{lem:SobCor} that for large $k$
\begin{equation*}
S\|\psi_R w_k\|_{L^{p^{\ast}}(\R^n)}^{p_{\max}} \leq \sum_{i=1}^n \frac{1}{p_i} \int_{\mathbb{R}^n} |D^{s_i}_{p_i}(\psi_R w_k)|^{p_i} \,\mathrm{d}x.
\end{equation*}

By \Cref{lem:klimit}, we have
\begin{equation*}
\lim_{k \to \infty} \int_{\mathbb{R}^n} |w_k|^{p_i} |D^{s_i}_{p_i} \psi_R|^{p_i} \, \mathrm{d}x = \lim_{k \to \infty} \int_{\mathbb{R}^n} |w_k|^{p_i} |D^{s_i}_{p_i} \phi_{R,0}|^{p_i} \, \mathrm{d}x = 0.
\end{equation*}

Therefore, by \eqref{eq:minkw},
\begin{equation*}
S \left( \lim_{R \to \infty} \limsup_{k \to \infty} \|\psi_R w_k\|_{p^{\ast}}^{p_{\max}} \right) \leq  \lim_{R \to \infty} \limsup_{k \to \infty} \sum_{i=1}^n \frac{1}{p_i} \int_{\mathbb{R}^n} |D^{s_i}_{p_i}w_k|^{p_i} \psi_R^{p_i} \,\mathrm{d}x,
\end{equation*}
and the desired result follows from \eqref{eq:nuChar} and \eqref{eq:muChar}.
\end{proof}

The following lemma provides the reverse H\"older inequalities for $\mu_i$ and $\nu$.

\begin{lemma}
\label{lem:revHolder}
Let $\phi \in C_c^{\infty}(\R^n)$ be such that $\supp(\phi) \subset M_1(0)$. Then, for every $\xi \in \R^n$ and $\eps > 0$,
\begin{align}
\label{eq:revHolder1}
\left( \int_{\R^n} \vert \phi_{\eps,\xi} \vert^{p^{\ast}} \, \d \nu\right)^{1/p^{\ast}} &\le C \sum_{i=1}^n \left(\int_{\R^n} \vert\phi_{\eps,\xi}\vert^{p_i} \, \d \mu_i \right)^{1/p_i} \quad\text{and}\\
\label{eq:revHolder2}
\left( \int_{\R^n} \vert \phi_{\eps,\xi} \vert^{p^{\ast}}  \, \d \nu\right)^{1/p^{\ast}} &\le C \left(\int_{\R^n} \vert\phi_{\eps,\xi}\vert^{p_{\max}}  \, \d \mu \right)^{1/p_{\max}}
\end{align} 
for some constant $C = C(\Vert \mu\Vert,\Vert \phi\Vert_{\infty}) >0$.
\end{lemma}

\begin{proof}
For simplicity, we write $\phi = \phi_{\eps,\xi}$. First, we prove \eqref{eq:revHolder1}. By definition of $\nu$ we deduce that $\int_{\R^n} \vert \phi w_k \vert^{p^{\ast}}  \, \d x \to \int_{\R^n} \vert \phi \vert^{p^{\ast}}  \, \d \nu$ as $k \to \infty$. By \Cref{thm:sobolev} it is sufficient to prove that for $i = 1,\dots,n$:
\begin{equation}
\label{eq:revHolderHelp1}
\limsup\limits_{k \to \infty} \Vert D^{s_i}_{p_i}(\phi w_k)\Vert_{L^{p_i}(\R^n)} \le C\left(\int_{\R^n} \vert\phi\vert^{p_i}  \, \d \mu_i \right)^{1/p_i}.
\end{equation}
By definition of $\mu_i$ it follows that $\int_{\R^n} \vert D^{s_i}_{p_i}(w_k) \phi \vert^{p_i}  \, \d x \to \int_{\R^n} \vert\phi\vert^{p_i}  \,\d \mu_i$, as $k \to \infty$. Therefore, by \eqref{eq:minkw}, it remains to show that
\begin{equation}
\label{eq:revHolderHelp2}
\int_{\R^n} \vert D^{s_i}_{p_i}(\phi)w_k \vert^{p_i}  \, \d x \to 0, ~~ \text{ as } k \to \infty,
\end{equation}
in order to deduce \eqref{eq:revHolderHelp1}. Since \eqref{eq:revHolderHelp2} follows from \Cref{lem:klimit}, estimate \eqref{eq:revHolder1} is proved.\\
For \eqref{eq:revHolder2} we observe that for every $i \in \{1,\dots,n\}$ by H\" older's inequality:
\begin{equation*}
\int_{\R^n} \vert\phi\vert^{p_i}  \, \d \mu_i \le C \int_{\R^n} \vert\phi\vert^{p_i}  \, \d \mu \le C \left(\Vert \mu \Vert\Vert\phi\Vert_{\infty}\right)^{1-\frac{p_i}{p_{\max}}} \left( \int_{\R^n} \vert \phi \vert^{p_{\max}}  \, \d \mu \right)^{p_i/p_{\max}}.
\end{equation*}
\end{proof}

\begin{corollary}
\label{cor:CC}
There exists an at most countable set $\Xi \subset \R^n$ and positive weights $(\nu_{\xi})_{\xi \in \Xi}$, $(\mu_{\xi})_{\xi \in \Xi}$ such that
\begin{equation*}
\nu = \sum_{\xi \in \Xi} \nu_{\xi}\delta_{\{\xi\}} \quad\text{and}\quad \mu \ge \sum_{\xi \in \Xi} \mu_{\xi}\delta_{\{\xi\}}.
\end{equation*}
\end{corollary}

\begin{proof}
By \Cref{lem:revHolder} it becomes clear that $\nu \ll \mu$, and for every $\xi \in \R^n$ and $\eps \in (0,1/2)$,
\begin{equation}
\label{eq:CChelp1}
\nu(M_{\eps}(\xi)) \le C \mu(M_{\eps}(\xi))^{\frac{p^{\ast}}{p_{\max}}}
\end{equation}
for some constant $C = C(\Vert \mu \Vert) > 0$. 
Since $\mu$ is a finite measure, the set $\Xi$ of atoms is at most countable. This proves already that $\mu \ge \sum_{\xi \in \Xi} \mu_{\xi}\delta_{\{\xi\}}$ and $\nu \ge \sum_{\xi \in \Xi} \nu_{\xi}\delta_{\{\xi\}}$ for two families of positive weights $(\mu_{\xi})_{\xi \in \Xi}$, $(\nu_{\xi})_{\xi \in \Xi}$ that are given by $\mu_{\xi} = \lim_{\eps \to 0} \mu(M_{\eps}(\xi))$ and $\nu_{\xi} = \lim_{\eps \to 0} \nu(M_{\eps}(\xi))$ for every $\xi \in \Xi$. By $\nu \ll \mu$, it follows that 
\begin{equation*}
\nu(M_{\eps}(\xi)) = \int_{M_{\eps}(\xi)} D_{\mu}(x)  \, \d \mu(x),
\end{equation*}
where 
\begin{equation*}
D_{\mu}(x) := \lim_{r \to 0} \frac{\nu(M_{r}(x))}{\mu(M_{r}(x))} \le C \lim_{r \to 0} \mu(M_{r}(x))^{\frac{p^{\ast}}{p_{\max}}-1} = \begin{cases}
C\mu_{x}^{\frac{p^{\ast}}{p_{\max}}-1},& ~~ x \in \Xi,\\
0,& ~~ \text{else},
\end{cases}
\end{equation*}
and we used \eqref{eq:CChelp1}. Consequently, $\nu = \sum_{\xi \in \Xi} \nu_{\xi}\delta_{\{\xi\}}$, as desired.
\end{proof}

As a consequence of \Cref{lem:revHolder} and \Cref{cor:CC}, we can infer a reverse H\"older inequality for the atoms of $\nu$ and $\widetilde{\mu}$.

\begin{corollary}
\label{cor:CC2}
Every $\xi \in \Xi$ is an atomic point of $\widetilde{\mu}$ and it holds
\begin{equation}
\label{eq:finallemmaHelp2}
S \nu_{\xi}^{\frac{p_{\max}}{p^{\ast}}} \le \widetilde{\mu}(\{\xi\}).
\end{equation}
In particular,
\begin{equation*}
\Vert \widetilde{\mu} \Vert \ge S \left(\sum_{\xi \in \Xi} \nu_{\xi}^{\frac{p_{\max}}{p^{\ast}}} +  \Vert v \Vert_{L^{p^{\ast}}(\R^n)}^{p_{\max}}\right).
\end{equation*}
\end{corollary}

\begin{proof}
In order to prove \eqref{eq:finallemmaHelp2}, we approximate $\widetilde{\mu}(\{\xi\})$ by cutoff functions $(\phi_{\eps,\xi})_{\eps}$, where $\phi \in C_c^{\infty}(\R^n)$ with $\supp(\phi) \subset M_1(0)$, $\phi(0) = 1$ and $0 \le \phi \le 1$. By Brezis–Lieb's lemma and \eqref{eq:conv_ptwise}, it holds that $\vert v_k\vert^{p^{\ast}} \,\d x \rightharpoonup \vert v \vert^{p^{\ast}} \,\d x + \nu$, and consequently:
\begin{equation}
\label{eq:finallemmaHelp3}
\lim_{\eps \to 0} \lim_{k \to \infty} \int_{\R^n} \phi_{\eps,\xi}^{p^{\ast}} \vert v_k \vert^{p^{\ast}} \,\d x = \lim_{\eps \to 0}\int_{M_{\eps}(\xi)} \phi_{\eps,\xi}^{p^{\ast}} \vert v \vert^{p^{\ast}} \,\d x + \lim_{\eps \to 0}\int_{M_{\eps}(\xi)} \phi_{\eps,\xi}^{p^{\ast}}  \,\d \nu \ge \nu_{\xi}.
\end{equation}
On the other hand, by \eqref{eq:minkw}, \Cref{lem:epslimit} and \Cref{lem:klimit}:
\begin{equation}
\label{eq:finallemmaHelp4}
\begin{aligned}
&\limsup_{\eps \to 0} \limsup_{k \to \infty}\left(\int_{\R^n} \vert D^{s_i}_{p_i}(\phi_{\eps,\xi}v_k)\vert^{p_i}  \,\d x\right)^{1/p_i}\\
&\le \limsup_{\eps \to 0} \limsup_{k \to \infty}\Bigg[\left(\int_{\R^n} \vert D^{s_i}_{p_i}(\phi_{\eps,\xi}) w_k\vert^{p_i}  \, \d x\right)^{1/p_i} + \left(\int_{\R^n} \vert D^{s_i}_{p_i}(\phi_{\eps,\xi})v\vert^{p_i}  \, \d x\right)^{1/p_i} \\
& \hspace*{4cm} + \left(\int_{\R^n} \vert D^{s_i}_{p_i}v_k\vert^{p_i}\vert \phi_{\eps,\xi} \vert^{p_i}  \, \d x \right)^{1/p_i}\Bigg]\\
&= \limsup_{\eps \to 0} \limsup_{k \to \infty} \left(\int_{\R^n} \vert D^{s_i}_{p_i}v_k\vert^{p_i}\vert \phi_{\eps,\xi} \vert^{p_i}  \, \d x \right)^{1/p_i}.
\end{aligned}
\end{equation}
Consequently, by combining \eqref{eq:finallemmaHelp3}, \eqref{eq:finallemmaHelp4}  and applying \Cref{lem:SobCor},
\begin{equation*}
\begin{split}
S \nu_{\xi}^{\frac{p_{\max}}{p^{\ast}}} &\le \lim_{\eps \to 0} \lim_{k \to \infty} S\left(\int_{\R^n} \phi_{\eps,\xi}^{p^{\ast}} \vert v_k\vert^{p^{\ast}}  \, \d x\right)^{\frac{p_{\max}}{p^{\ast}}}\\
&\le \limsup_{\eps \to 0} \limsup_{k \to \infty} \sum_{i =1}^n \frac{1}{p_i} \int_{\R^n} \vert D^{s_i}_{p_i}(\phi_{\eps,\xi}v_k)\vert^{p_i}  \, \d x\\
&\le \limsup_{\eps \to 0} \limsup_{k \to \infty} \sum_{i =1}^n \frac{1}{p_i} \int_{\R^n} \vert D^{s_i}_{p_i}v_k\vert^{p_i}\vert \phi_{\eps,\xi} \vert^{p_i} \, \d x\\
&\le \limsup_{\eps \to 0} \widetilde{\mu}(M_{\eps}(\xi)) \\
&=  \widetilde{\mu}(\{\xi\}).
\end{split}
\end{equation*}
Hence, \eqref{eq:finallemmaHelp2} holds true. From the observation that $\sum_{i =1}^n \frac{1}{p_i} \vert D^{s_i}_{p_i} v\vert^{p_i}\d x$ is orthogonal to the atomic part of $\widetilde{\mu}$, we conclude that
\begin{equation*}
\Vert \widetilde{\mu} \Vert \ge S \sum_{\xi \in \Xi} \nu_{\xi}^{\frac{p_{\max}}{p^{\ast}}} + \sum_{i =1}^n \frac{1}{p_i} \Vert D_{p_i}^{s_i}v\Vert_{L^{p_i}(\R^n)}^{p_i} \ge S \left(\sum_{\xi \in \Xi} \nu_{\xi}^{\frac{p_{\max}}{p^{\ast}}} + \Vert v \Vert_{L^{p^{\ast}}(\R^n)}^{p_{\max}} \right),
\end{equation*}
where we also applied \Cref{lem:SobCor}. This concludes the proof.
\end{proof}

\subsection{Proof of main results}\label{ssec:main}

This subsection is devoted to the proofs of the main results \Cref{thm:existence} and \Cref{cor:existence}. We first prove \Cref{thm:unit_norm}.

\begin{proof} [Proof of \Cref{thm:unit_norm}]
Let us assume by contradiction that $\Vert v \Vert_{L^{p^{\ast}}(\R^n)} < 1$.
First, we have that
\begin{equation}\label{eq:unit_normHelp1}
\begin{aligned}
1 = \frac{1}{S} \left(\Vert \widetilde{\mu}\Vert + \mu_{\infty}\right) & \ge \Vert v \Vert_{L^{p^{\ast}}(\R^n)}^{p_{\max}} + \sum_{\xi\in\Xi} \nu_{\xi}^{\frac{p_{\max}}{p^{\ast}}} + \nu_{\infty}^{\frac{p_{\max}}{p^{\ast}}} \\
& \ge \Vert v \Vert_{L^{p^{\ast}}(\R^n)}^{p_{\max}} + \Vert \nu \Vert^{\frac{p_{\max}}{p^{\ast}}} + \nu_{\infty}^{\frac{p_{\max}}{p^{\ast}}},
\end{aligned}
\end{equation}
where the second step follows from \Cref{lem:inftyrevHolder} and \Cref{cor:CC2}. In order to prove the first equality in \eqref{eq:unit_normHelp1} we define $\psi_R : \R^n \to \R$ as in the proof of \Cref{lem:inftyrevHolder} and compute
\begin{equation*}
\begin{split}
S &= \lim_{k \to \infty} \sum_{i=1}^n\frac{1}{p_i} \Vert D_{p_i}^{s_i}v_k \Vert_{L^{p_i}(\R^n)}^{p_i}\\
&= \lim_{R \to \infty} \lim_{k \to \infty} \int_{\R^n} \sum_{i=1}^n \frac{1}{p_i}\vert D^{s_i}_{p_i}v_k\vert^{p_i}(1-\psi_R^{p_i}) \,  \d x + \lim_{R \to \infty} \limsup_{k \to \infty} \int_{\R^n} \sum_{i=1}^n \frac{1}{p_i}\vert D^{s_i}_{p_i}v_k\vert^{p_i} \psi_R^{p_i} \, \d x\\
&= \lim_{R \to \infty} \int_{\R^n} (1-\psi_R^{p_i}) \, \d \widetilde{\mu} + \lim_{R \to \infty} \limsup_{k \to \infty} \int_{\R^n} \sum_{i=1}^n \frac{1}{p_i}\vert D^{s_i}_{p_i}w_k\vert^{p_i} \psi_R^{p_i} \, \d x\\
&= \Vert \widetilde{\mu} \Vert + \mu_{\infty},
\end{split}
\end{equation*}
where we used \eqref{eq:muChar}. This proves \eqref{eq:unit_normHelp1}.
On the other hand, it holds that
\begin{equation}
\label{eq:unit_normHelp2}
\Vert v \Vert_{L^{p^{\ast}}(\R^n)}^{p_{\max}} + \Vert \nu \Vert^{\frac{p_{\max}}{p^{\ast}}} + \nu_{\infty}^{\frac{p_{\max}}{p^{\ast}}} \ge 1.
\end{equation}
This follows from the observation
\begin{equation}
\label{eq:unit_normHelp3}
\begin{split}
1 &= \lim_{k \to \infty} \Vert v_k \Vert_{L^{p^{\ast}}(\R^n)}^{p^{\ast}} = \lim_{R \to \infty} \lim_{k \to \infty} \int_{\R^n} (1-\psi_R^{p^{\ast}}) \vert v_k \vert^{p^{\ast}} \, \d x + \lim_{R \to \infty} \lim_{k \to \infty} \int_{\R^n} \psi_R^{p^{\ast}} \vert v_k \vert^{p^{\ast}} \, \d x\\
&=\Vert v \Vert_{L^{p^{\ast}}(\R^n)}^{p^{\ast}} + \Vert \nu \Vert+ \nu_{\infty},
\end{split}
\end{equation}
where we used that $\vert v_k \vert^{p^{\ast}}\, \d x \rightharpoonup \vert v \vert^{p^{\ast}} \, \d x + \nu$ as a consequence of Brezis–Lieb's lemma and \eqref{eq:conv_ptwise}, and \eqref{eq:nuinftyrep1}. In order to derive \eqref{eq:unit_normHelp2}, we take \eqref{eq:unit_normHelp3} to the power $p_{\max}/p^{\ast} < 1$.
Combining \eqref{eq:unit_normHelp1} and \eqref{eq:unit_normHelp2} yields that
\begin{equation*}
\Vert v \Vert_{L^{p^{\ast}}(\R^n)}^{p_{\max}} + \Vert \nu \Vert^{\frac{p_{\max}}{p^{\ast}}} + \nu_{\infty}^{\frac{p_{\max}}{p^{\ast}}} = 1.
\end{equation*}
Consequently, $\Vert v \Vert_{L^{p^{\ast}}(\R^n)}^{p^{\ast}}, \Vert \nu \Vert, \nu_{\infty} \in \{0,1\}$. By assumption, we can now deduce that $v = 0$. Moreover, by \eqref{eq:sup} we know that $\nu_{\infty} \le 1/2$ and therefore it must be that 
\begin{equation*}
1 = \Vert \nu \Vert = \sum_{\xi \in \Xi} \nu_{\xi},
\end{equation*}
where we applied \Cref{cor:CC} in the last step.
From \eqref{eq:unit_normHelp1}, we deduce that $1 \ge  \sum_{\xi\in\Xi} \nu_{\xi}^{\frac{p_{\max}}{p^{\ast}}}$. It follows that
\begin{equation*}
1 = \left( \sum_{\xi \in \Xi} \nu_{\xi} \right)^{\frac{p_{\max}}{p^{\ast}}} \le \sum_{\xi \in \Xi} \nu_{\xi}^{\frac{p_{\max}}{p^{\ast}}} \le 1,
\end{equation*}
and therefore all inequalities in the above line can be replaced by equality signs. Thus, $\nu_{\xi} \in \{0,1\}$ for every $\xi \in \Xi$ and therefore, there exists exactly one $\xi_0 \in \Xi$ such that $\nu_{\xi_0} = 1$ and $\nu_{\xi} = 0$ for every $\xi \in \Xi \setminus \{\xi_0\}$. This leads to a contradiction since by \eqref{eq:sup}:
\begin{equation*}
\frac{1}{2} = \sup_{y \in \R^n} \int_{M_1(y)} \vert v_k \vert^{p^{\ast}} \, \d x \ge \sup_{y \in \R^n} \int_{M_1(\xi_0)} \vert v_k \vert^{p^{\ast}} \, \d x \to \Vert \nu \Vert = 1.
\end{equation*}
Therefore, it must hold that $\Vert v \Vert_{L^{p^{\ast}}(\R^n)} = 1$, as desired.
\end{proof}

We are now ready to prove \Cref{thm:existence}.

\begin{proof} [Proof of \Cref{thm:existence}]
As explained in the beginning of \Cref{sec:CCP}, let us take a minimizing sequence $(u_k) \subset \Dsp$ for $S$ satisfying $\|u_k\|_{L^{p^{\ast}}(\mathbb{R}^n)} = 1$ for every $k \in \mathbb{N}$. We may assume that $u_k$ is nonnegative because $(|u_k|)$ is also a minimizing sequence for $S$. Indeed, $\|D^{s_i}_{p_i} |u_k| \|_{L^{p_i}(\mathbb{R}^n)} \leq \|D^{s_i}_{p_i}u_k \|_{L^{p_i}(\mathbb{R}^n)}$ follows from the inequality
\begin{equation*}
||u(x)|-|u(x+he_i)|| \leq |u(x)-u(x+he_i)|.
\end{equation*}
Let us now consider the rescaled function $v_k$ as in the beginning of \Cref{sec:CCP}, which satisfies \eqref{eq:sup}, \eqref{eq:conv_Dsp}, and \eqref{eq:conv_ptwise}. Then, it follows from \eqref{eq:lower_semicont} and \Cref{thm:unit_norm} that $v$ is a nontrivial minimizer for $S$. Moreover, since $v_k$ is nonnegative a.e. in $\mathbb{R}^n$, so is $v$.
\end{proof}

Let us next prove \Cref{cor:existence}.

\begin{proof} [Proof of \Cref{cor:existence}]
By the standard Lagrange multiplier rule, any minimizer $u \in \mathcal{D}^{\vec{s},\vec{p}}(\R^n)$ of \eqref{eq:min_prob} satisfies
\begin{equation*}
\sum_{i=1}^n s_i(1-s_i) \int_{\mathbb{R}} \frac{|u(x)-u(x+he_i)|^{p_i-2}(u(x)-u(x+he_i))}{|h|^{1+s_ip_i}} \,\mathrm{d}h = S|u|^{p^{\ast}-2}u \quad \text{in}~ \mathbb{R}^n.
\end{equation*}
Thus, the function $v$ defined by $v(x) = u(S^{-\frac{1}{s_1p_1}}x_1, \dots, S^{-\frac{1}{s_np_n}}x_n)$ solves \eqref{eq:main_eq}.
\end{proof}

By rescaling a solution of \eqref{eq:main_eq}, we get infinitely many solutions.

\begin{corollary}
If $u \in \mathcal{D}^{\vec{s},\vec{p}}(\R^n)$ solves \eqref{eq:main_eq}, then the function $v$ defined by \eqref{eq:v} also solves \eqref{eq:main_eq}.
\end{corollary}

\end{document}